\documentclass[11pt]{amsart}
\usepackage[margin=1.35in]{geometry}
\usepackage{amscd,amsmath,dsfont,amsxtra,amsthm,amssymb,stmaryrd,xr,mathrsfs,mathtools,enumerate,commath, comment}
\usepackage{stmaryrd}
\usepackage{xcolor}
\usepackage{commath}
\usepackage{comment}
\usepackage{tikz-cd}
\usepackage{longtable} 
\usepackage{pdflscape} 
\usepackage{booktabs}
\usepackage{hyperref}
\definecolor{vegasgold}{rgb}{0.77, 0.7, 0.35}
\definecolor{darkgoldenrod}{rgb}{0.72, 0.53, 0.04}
\definecolor{gold(metallic)}{rgb}{0.83, 0.69, 0.22}
\hypersetup{
 colorlinks=true,
 linkcolor=darkgoldenrod,
 filecolor=brown,      
 urlcolor=gold(metallic),
 citecolor=darkgoldenrod,
 }

\usepackage[all,cmtip]{xy}

\DeclareFontFamily{U}{wncy}{}
\DeclareFontShape{U}{wncy}{m}{n}{<->wncyr10}{}
\DeclareSymbolFont{mcy}{U}{wncy}{m}{n}
\DeclareMathSymbol{\Sh}{\mathord}{mcy}{"58}
\usepackage[T2A,T1]{fontenc}
\usepackage[OT2,T1]{fontenc}

\newtheorem{theorem}{Theorem}[section]

\newtheorem*{theorem*}{Theorem}
\newtheorem*{ass*}{Assumption}
\newtheorem{definition}{Definition}[section]
\newtheorem{corollary}[theorem]{Corollary}
\newtheorem{remark}[theorem]{Remark}

\newtheorem{proposition}[theorem]{Proposition}
\newtheorem{question}[theorem]{Question}

\newcommand{\dL}{\mathds{L}}

\newcommand{\cF}{\mathcal{F}}

\newcommand{\cG}{\mathcal{G}}

\newcommand{\Z}{\mathbb{Z}}
\newcommand{\p}{\mathfrak{p}}

\newcommand{\Q}{\mathbb{Q}}
\newcommand{\F}{\mathbb{F}}

\newcommand{\cL}{\mathcal{L}}

\newcommand{\cO}{\mathcal{O}}

\newcommand{\cP}{\mathcal{P}}

\newcommand{\op}[1]{\operatorname{#1}}

 \newcommand{\widebar}[1]{\mkern 2.5mu\overline{\mkern-2.5mu#1\mkern-2.5mu}\mkern 2.5mu}

\numberwithin{equation}{section}

\begin{document}

\title[Asymptotic Growth patterns for class field towers]{Asymptotic Growth patterns for class field towers}

\author[A.~Bhattacharyya]{Arindam Bhattacharyya}
\address[A.~Bhattacharyya]{Chennai Mathematical Institute, H1, SIPCOT IT Park, Kelambakkam, Siruseri, Tamil Nadu 603103, India}
\email{arindamb@cmi.ac.in}

\author[V.~Kadiri]{Vishnu Kadiri}
\address[V.~Kadiri]{Chennai Mathematical Institute, H1, SIPCOT IT Park, Kelambakkam, Siruseri, Tamil Nadu 603103, India}
\email{vishnuk@cmi.ac.in}

\author[A.~Ray]{Anwesh Ray}
\address[A.~Ray]{Chennai Mathematical Institute, H1, SIPCOT IT Park, Kelambakkam, Siruseri, Tamil Nadu 603103, India}
\email{anwesh@cmi.ac.in}

\maketitle

\begin{abstract}
Let $p$ be an odd prime number. We study growth patterns associated with finitely ramified Galois groups considered over the various number fields varying in a $\mathbb{Z}_p$-tower. These Galois groups can be considered as non-commutative analogues of ray class groups. For certain $\mathbb{Z}_p$-extensions in which a given prime above $p$ is completely split, we prove precise asymptotic lower bounds. Our investigations are motivated by the classical results of Iwasawa, who showed that there are growth patterns for $p$-primary class numbers of the number fields in a $\mathbb{Z}_p$-tower. 
\end{abstract}

\section{Introduction}

\subsection{Motivation from Iwasawa theory} Let $\dL$ be a number field and $p$ be an odd prime number. Choose an algebraic closure $\bar{\dL}$ of $\dL$. Let $H_p(\dL)$ be the maximal abelian unramified $p$-extension of $\dL$, and let $\mathcal{A}_p(\dL)$ denote the Galois group $\op{Gal}(H_p(\dL)/\dL)$. By class field theory, $\mathcal{A}_p(\dL)$ is naturally isomorphic to the $p$-primary part of the class group of $\dL$. Setting $h_p(\dL):=\# \mathcal{A}_p(\dL)$, we refer to $h_p(\dL)$ as the $p$-class number of $\dL$. We set $\Z_p$ to denote the ring of $p$-adic integers. A $\Z_p$-extension of a number field $\dL$ is an infinite Galois extension $\dL_\infty/\dL$, such that $\op{Gal}(\dL_\infty/\dL)$ is isomorphic to $\Z_p$ (as a topological group). Given a $\Z_p$-extension $\dL_\infty/\dL$, and $n\in\Z_{\geq 0}$, we set $\dL_n/\dL$ to be the extension such that $\dL_n\subseteq \dL_\infty$ and $[\dL_n:\dL]=p^n$. The field $\dL_n$ is the \emph{$n$-th layer}, and we obtain a tower of number fields 
\[\dL=\dL_0\subseteq \dL_1\subseteq \dots \subseteq \dL_n\subseteq \dL_{n+1}\subseteq \dots.\]
Let $e_n\in \Z_{\geq 0}$ be such that $p^{e_n}=h_p(\dL_n)$. In his seminal work, Iwasawa \cite{iwasawa1973zl} showed that there are constants $\mu, \lambda\in \Z_{\geq 0}$ and $\nu\in \Z$, such that for all large enough values of $n$, one has that $e_n=p^n \mu+ n\lambda+\nu$.
\par Thus, Iwasawa's results show that there are interesting growth patterns for $p$-class numbers in certain infinite towers of number fields. These results motivate the study of asymptotic growth properties of $p$ Hilbert class field towers along $\Z_p$-extensions. Throughout, $p$ will be a fixed odd prime number, and let $\dL_\infty/\dL$ be a $\Z_p$-extension. Set $\cF(\dL_n)$ to denote the maximal unramified pro-$p$ extension of $\dL_n$, and set $G_n:=\op{Gal}(\cF(\dL_n)/\dL_n)$. We identify the abelianization of $G_n$ with \[\mathcal{A}_p(\dL_n)=\op{Gal}(H_p(\dL_n)/\dL_n).\]

\par The field $\cF(\dL_n)$ contains a $p$ Hilbert class field tower over $\dL_n$. In greater detail, for $j\in \Z_{\geq 0}$, define $H_p^{(j)}(\dL_n)$ inductively as follows
\[\begin{split}& H_p^{(0)}(\dL_n):=\dL_n,\\ & H_p^{(1)}(\dL_n):=H_p(\dL_n), \\ 
& H_p^{(j)}(\dL_n):=H_p\left(H_p^{(j-1)}(\dL_n)\right)\text{ for }j\geq 2. \end{split}\] In this way, we obtain a tower of $p$-extensions of $\dL_n$ whose union is equal to $\cF(\dL_n)$. The field $\cF(\dL_n)$ is infinite if and only if $H_p^{(j)}(\dL_n)\subsetneq H_p^{(j+1)}(\dL_n)$ for all $j\geq 1$. These Hilbert class field towers, when viewed along the $\Z_p$-extension, are represented by the following diagram
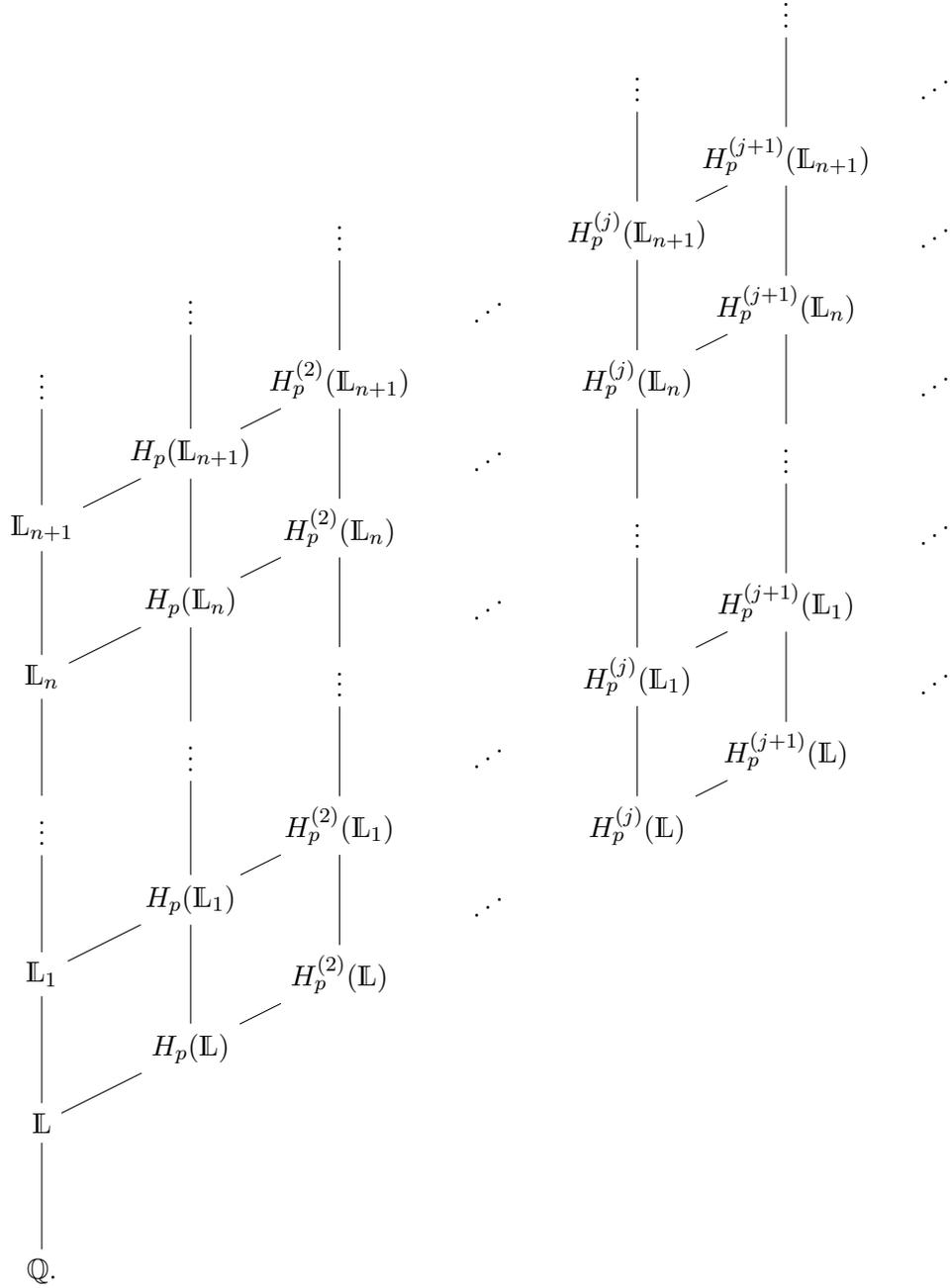
\[ \begin{tikzpicture}
\node (Q) at (-3,-2) {$\Q$.};
    \node (Q1) at (-3,0) {$\dL$};
    \node (Q2) at (-3,2) {$\dL_1$};
     \node (Q3) at (-3,4) {$\vdots$};
      \node (Q4) at (-3,6) {$\dL_n$};
       \node (Q5) at (-3,8) {$\dL_{n+1}$};
        \node (Q6) at (-3,10) {$\vdots$};

     \node (P1) at (-1,1) {$H_p(\dL)$};
    \node (P2) at (-1,3) {$H_p(\dL_1)$};
     \node (P3) at (-1,5) {$\vdots$};
      \node (P4) at (-1,7) {$H_p(\dL_n)$};
       \node (P5) at (-1,9) {$H_p(\dL_{n+1})$};
        \node (P6) at (-1,11) {$\vdots$};

       \node (R1) at (1,2) {$H_p^{(2)}(\dL)$};
    \node (R2) at (1,4) {$H_p^{(2)}(\dL_1)$};
     \node (R3) at (1,6) {$\vdots$};
      \node (R4) at (1,8) {$H_p^{(2)}(\dL_n)$};
       \node (R5) at (1,10) {$H_p^{(2)}(\dL_{n+1})$};
        \node (R6) at (1,12) {$\vdots$};  

         \node (S1) at (3,3) {\reflectbox{$\ddots$}};
    \node (S2) at (3,5) {\reflectbox{$\ddots$}};
     \node (S3) at (3,7) {\reflectbox{$\ddots$}};
      \node (S4) at (3,9) {\reflectbox{$\ddots$}};
       \node (S5) at (3,11) {\reflectbox{$\ddots$}};
        \node (S6) at (3,13) {};  

         \node (T1) at (5,4) {$H_p^{(j)}(\dL)$};
    \node (T2) at (5,6) {$H_p^{(j)}(\dL_1)$};
     \node (T3) at (5,8) {$\vdots$};
      \node (T4) at (5,10) {$H_p^{(j)}(\dL_n)$};
       \node (T5) at (5,12) {$H_p^{(j)}(\dL_{n+1})$};
        \node (T6) at (5,14) {$\vdots$};  

          \node (U1) at (7,5) {$H_p^{(j+1)}(\dL)$};
    \node (U2) at (7,7) {$H_p^{(j+1)}(\dL_1)$};
     \node (U3) at (7,9) {$\vdots$};
      \node (U4) at (7,11) {$H_p^{(j+1)}(\dL_n)$};
       \node (U5) at (7,13) {$H_p^{(j+1)}(\dL_{n+1})$};
        \node (U6) at (7,15) {$\vdots$};

                 \node (V1) at (9,6) {\reflectbox{$\ddots$}};
    \node (V2) at (9,8) {\reflectbox{$\ddots$}};
     \node (V3) at (9,10) {\reflectbox{$\ddots$}};
      \node (V4) at (9,12) {\reflectbox{$\ddots$}};
       \node (V5) at (9,14) {\reflectbox{$\ddots$}};
        \node (V6) at (9,16) {};  

 \draw (Q)--(Q1) {};
    \draw (Q1)--(Q2) {};
     \draw (Q2)--(Q3){};
    \draw (Q3)--(Q4) {};
 \draw (Q4)--(Q5) {};
  \draw (Q5)--(Q6) {};

   \draw (P1)--(P2) {};
     \draw (P2)--(P3){};
    \draw (P3)--(P4) {};
 \draw (P4)--(P5) {};
  \draw (P5)--(P6) {};

   \draw (R1)--(R2) {};
     \draw (R2)--(R3){};
    \draw (R3)--(R4) {};
 \draw (R4)--(R5) {};
  \draw (R5)--(R6) {};

   \draw (T1)--(T2) {};
     \draw (T2)--(T3){};
    \draw (T3)--(T4) {};
 \draw (T4)--(T5) {};
  \draw (T5)--(T6) {};

   \draw (U1)--(U2) {};
     \draw (U2)--(U3){};
    \draw (U3)--(U4) {};
 \draw (U4)--(U5) {};
  \draw (U5)--(U6) {};

\draw (Q1)--(P1) {};
\draw (P1)--(R1) {};
\draw (T1)--(U1) {};

\draw (Q2)--(P2) {};
\draw (P2)--(R2) {};
\draw (T2)--(U2) {};

\draw (Q4)--(P4) {};
\draw (P4)--(R4) {};
\draw (T4)--(U4) {};

\draw (Q5)--(P5) {};
\draw (P5)--(R5) {};
\draw (T5)--(U5) {};
  
    \end{tikzpicture}\]

\par Iwasawa studied asymptotic growth patterns for the degrees $[H_p(\dL_n):\dL_n]$ considered along the tower that is the second column (from the left margin) in the above diagram. In the spirit of Iwasawa theory, we consider natural growth questions for the pro-$p$ groups $G_{n}$ as $n\rightarrow \infty$. We define the \emph{exponential growth number} $\rho(G)$ of a finitely generated pro-$p$ group $G$, which is a natural invariant associated with its Hilbert series. In this context, there is a natural analogy with the Hilbert series of an algebraic variety. Let $\Omega(G)$ denote the mod-$p$ Iwasawa algebra associated to $G$, defined as the inverse limit $\varprojlim_U \F_p[G/U]$, where $U$ ranges over all finite index normal subgroups of $G$. Let $I_G$ be the augmentation ideal of $\Omega(G)$, and for $n\geq 0$, set $c_n(G):=\op{dim}_{\F_p}\left(I_G^n/I_G^{n+1}\right)$, where it is understood that $c_0(G):=1$. Let \[H(G;t)=\sum_{n\geq 0} c_n(G) t^n\] denote the Hilbert series associated to $G$, cf. Definition \ref{def of Hilbert series}. If $G$ is finite, then $c_n(G)=0$ for large enough values of $n$. When $G$ is a $p$-adic analytic group of dimension $d>0$, we find that $c_n(G)=O(n^d)$ (cf. \cite[section 12.3, p. 307]{dixonetal}). On the other hand, if a group $G$ has infinite rank, the numbers $c_n(G)$ grow at an exponential rate (cf. Proposition 12.17 of \emph{loc. cit.}). In practice, the Galois groups that arise from infinite $p$ Hilbert class field towers may not be $p$-adic analytic groups. For many of the Galois groups constructed in this article, the numbers $c_n(G)$ increase at an exponential rate. The radius of convergence of $H(G;t)$ is given by $R_G=\rho(G)^{-1}$, where $\rho(G):=\limsup_{n\rightarrow \infty} |c_n(G)|^{\frac{1}{n}}$. It is understood that $R_G:=\infty$ when $\rho(G)=0$, and in this case, $H(G, t)$ is an \emph{entire function}. The constant $\rho(G)$ measures the exponential growth of subgroups in the Zassenhaus filtration of $G$. It is clear that if $\rho(G)>1$, then $G$ is not an analytic pro-$p$ group. It is shown, in various contexts that pro-$p$ Hilbert class field towers are infinite via an application of the Golod-Shafarevich-Vinberg criterion (cf. \cite[Theorem 1.2]{hajir2020infinite}). This strategy is developed and employed in the following works: \cite{venkov1978p, schoof1986infinite, kisilevsky1987sufficient, joshi2011infinite, hajir2020infinite, hajir2021cutting}. The class of pro-$p$ groups for which it can be shown that $\rho(G)>1$, via the Golod-Shafarevich-Vinberg criterion have special properties, and are known as Golod-Shafarevich groups (cf. Definition \ref{def of GS groups}). We refer to \cite{ershov2012golod} for an introduction to the theory of Golod-Shafarevich groups. 

\par Given a number field $L$, its \emph{root discriminant} is defined to be $D_L^{1/n}$, where $D_L$ is its absolute discriminant, and $n=[L:\Q]$. It is an old question as to whether there exists an infinite tower of number fields, unramified away from a finite set of primes, for which the root discriminant is bounded. Since the root discriminant is constant in unramified extensions, the above question is related to the constants of Martinet \cite{martinet1978tours} and Odlyzko's bounds \cite{odlyzko1990bounds}. In the number field extensions $L/\dL_n$ such that $L\subset \cF(\dL_n)$, the root discriminant remains bounded. Hajir and Maire \cite{hajir2001tamely} construct tamely ramified Golod-Shafarevich Galois groups and are able to improve upon Martinet's constants. These root discriminant bounds are further refined by Hajir, Maire and Ramakrishna in \cite{hajir2021cutting}. One is thus interested in constructing infinite unramified Galois pro-$p$ groups $G$, such that $\rho(G)$ is large. In this paper, we study the following question.
\begin{question}
Let $\dL$ be a number field and $p$ be a prime number. Given a $\Z_p$-extension $\dL_\infty/\dL$, what can one say about the growth of $\rho(G_{n})$, as $n\rightarrow \infty$?
\end{question}
\subsection{Main results} We consider a variant of the above question, for certain natural $\Z_p$-extensions in which one of the primes above $p$ is infinitely split. For $k\geq 0$, let $\cF^{[k]}(\dL_n)$ be the maximal pro-$p$ extension of $\dL_n$ for which
\begin{itemize}
\item all primes $v\nmid p$ are unramified, 
\item all decomposition groups at primes $v|p$ are abelian, 
\item for all primes $v\mid p$, every element of the decomposition group of $v$ has order dividing $p^k$.
\end{itemize}
Note that $\cF(\dL_n)\supseteq \cF^{[0]}(\dL_n)$. For $k\geq 1$, there is finite ramification in the extension $\cF^{[k]}(\dL_n)/\dL_n$. As a consequence, there is an intermediate extension $\dL_n\subseteq \dL_n'\subseteq \cF^{[k]}(\dL_n)$, such that $\dL_n'/\dL_n$ is a finite extension, and $\cF^{[k]}(\dL_n)/\dL_n'$ is unramified (cf. the proof of \cite[Proposition 1.5]{hajir2020infinite}). In particular, for all number fields contained in $\cF^{[k]}(\dL_n)$, the root discriminant is bounded. We shall set $\cG^{[k]}(\dL_n):=\op{Gal}(\cF^{[k]}(\dL_n)/\dL_n)$, and $\rho^{[k]}(\dL_n):=\rho\left(\cG^{[k]}(\dL_n)\right)$.

\par Let $\dL$ be a number field extension of $\Q$ and let $p$ be an odd prime. Let $r_1$ (resp. $r_2$) be the number of real embeddings (resp. complex embeddings) of $\dL$. Let $\p=\p_1, \p_2, \dots, \p_g$ be the primes of $\dL$ that lie above $p$, and assume that $g\geq 2$. Set $e_i:=e(\p_i/p)$ (resp. $f_i:=f(\p_i/p)$) to denote the ramification index (resp. inertial degree) of $\p_i$ over $p$. Note that $d_i:=[\dL_{\p_i}:\Q_p]=e_i f_i$. Let $U_i:=\cO_{\dL_{\p_i}}^\times$ and $U_i^{(1)}$ the principal local units in $U_i$. We set $\mathfrak{U}$ to denote the product $\prod_{i=2}^g U_i^{(1)}$, and $\bar{E}$ to denote the $p$-adic closure of the image of the group of global units $\cO_{\dL}^\times$ in $\mathfrak{U}$. Setting $\delta:=\op{rank}_{\Z_p}(\bar{E})$, we note that $\delta \leq \op{rank} \cO_{\dL}^\times =r_1+r_2-1$. We set \[m:=\op{rank}_{\Z_p}\left(\mathfrak{U}/\bar{E}\right)-1=[\dL:\Q]-d_1-\delta-1,\] and assume that $m\geq 1$. It follows from a standard application of global class field theory that there exists a $\Z_p$-extension $\dL_\infty/\dL$ in which $\p$ is completely split. Note that when $\dL$ is totally imaginary, the condition $m\geq 1$ is automatically satisfied when $[\dL:\Q]\geq 2(d_1+1)$. 

\begin{theorem}\label{thm 1.2}
Assume that the following conditions are satisfied
\begin{enumerate}
\item $p$ is odd and there are $g>1$ primes of $\dL$ that lie above $p$,
 \item\label{p1 of thm 1.2} $\dL$ contains $\mu_p$, the $p$-th roots of unity,
    \item\label{p2 of thm 1.2} $[\dL:\Q]
    \geq 2\left(d_1+1\right)$,
    \item\label{p4 of thm 1.2} $([\dL:\Q]+2)^2> 8\left(\sum_{i=2}^g d_i^2\right)$.
\end{enumerate} 
 Then, there exists a $\Z_p$-extension $\dL_\infty/\dL$ in which $\p_1$ is totally split. Moreover, there exists a constant $C>0$ (independent of $n$ and $k$) and $n_0\in \Z_{\geq 0}$, such that for all $n\geq n_0$ and $k\geq 1$, we have that
\[\rho^{[k]}(\dL_n)\geq Cp^n.\]
Let $\mu(\dL_\infty/\dL)$ denote the Iwasawa $\mu$-invariant for the $\Z_p$-extension $\dL_\infty/\dL$. Let $T_1$ be the number of primes of $\dL$ that lie above $p$, that are infinitely decomposed in $\dL_\infty$.
The constant $C$ can be chosen to be any value such that
\[0< C < \frac{2\left(\sum_{i=2}^g d_i^2\right)}{\left(2 T_1+[\dL:\Q]+2\mu(\dL_\infty/\dL)\right)}.\]
\end{theorem}

\begin{remark}
    \begin{itemize}
        \item We observe that the hypotheses imply that 
\[\frac{2\left(\sum_{i=2}^g d_i^2\right)}{\left(2T_1+[\dL:\Q]+2\mu(\dL_\infty/\dL)\right)}< \left(\frac{[\dL:\Q]+2}{4}\right).\]
\item Note that the condition \eqref{p1 of thm 1.2} implies that $\dL$ is totally imaginary and \eqref{p2 of thm 1.2} implies that $m\geq 1$.
    \end{itemize}
\end{remark}

 We prove Theorem \ref{thm 1.2} by adapting the strategy of Hajir, Maire and Ramakrishna \cite{hajir2020infinite}. The following result illustrates Theorem \ref{thm 1.2}.

\begin{corollary}\label{main cor}
Let $p\geq 3$ and $\ell\geq 11$ be distinct primes and set $\dL:=\Q(\mu_{p\ell})$. Let $\p$ be any prime of $\dL$ that lies above $p$. Assume that there are $g=(\ell-1)$ primes of $\dL$ that lie above $p$, i.e., $\ell\equiv 1\mod{p}$. Then, there exists a $\Z_p$-extension $\dL_\infty/\dL$ in which $\p$ is totally split. Let $T_1$ be the number of primes of $\dL$ that lie above $p$, that are infinitely decomposed in $\dL_\infty$. For any constant \[0<C<\frac{2(\ell-2)(p-1)^2}{\left( 2 T_1 +(p-1)(\ell-1)+2\mu(\dL_\infty/\dL) \right)},\] we have that 
\[\rho^{[k]}(\dL_n)\geq Cp^n\] for all $k\geq 1$, and all large enough values of $n$.
\end{corollary}

We also define another notion measuring the \emph{size} $m(G)$ of a Golod-Shafarevich group $G$ (cf. Definition \ref{defn size}). We prove an analogous result for the growth of the numbers $m^{[k]}(\dL_n):=m\left(\cG^{[k]}(\dL_n)\right)$, cf. Theorem \ref{thm 3.6}.

\subsection{Organization} Including the introduction, the manuscript consists of three sections. In section \ref{s 2}, we introduce preliminary notions. In greater detail, we develop the Golod-Shafarevich theory of pro-$p$ groups. We prove an explicit criterion (cf. Proposition \ref{lower bound on rho(G)}) which gives an explicit lower bound for the exponential growth number $\rho(G)$ for a Golod-Shafarevich group $G$. In section \ref{s 3}, we apply the results from section \ref{s 2} to prove Theorem \ref{thm 1.2}, Corollary \ref{main cor} and Theorem \ref{thm 3.6}.

\subsection{Acknowledgment} This project was started when the third author was at the Centre de recherches mathematiques, Montreal. At the time, the third author's research was supported by the CRM-Simons fellowship. We are very thankful to Katharina M\"uller and Ravi Ramakrishna for numerous helpful comments. We thank the anonymous referee for the excellent report. 

\section{Pro-$p$ groups and their Hilbert series}\label{s 2}
\par We review some preliminaries and set up some basic notation in this section. Throughout, $p$ shall be an odd prime. Let $G$ be a finitely generated pro-$p$ group, set $\Omega(G)$ to denote the completed group algebra of $G$ over $\F_p$. More precisely, $\Omega(G)$ is defined to be the inverse limit
\[\Omega(G):=\varprojlim_U \F_p[G/U],\] where $U$ runs over all finite index normal subgroups of $G$. We refer to $\Omega(G)$ as the \emph{mod-$p$ Iwasawa algebra} of $G$. Many properties of the group $G$ are captured by algebraic properties of $\Omega(G)$. We consider the \emph{Hilbert series} associated with $\Omega(G)$. Given a normal subgroup $U$ of $G$, there is a natural map $\iota_U:G\rightarrow \F_p[G/U]$, which sends $g\in G$ to the element $\bar{g}\cdot 1$. Here, $\bar{g}$ is the congruence class of $g$ modulo $U$. The map $\iota:G\rightarrow \Omega(G)$ is the inverse limit $\iota:=\varprojlim_U \iota_U$. The augmentation map $\alpha:\Omega(G)\rightarrow \F_p$ maps each element of the form $\iota(g)$ to $1$. We shall, by abuse of notation, simply let $g\in \Omega(G)$ simply denote the element $\iota(g)$. 
 Set $I_G$ to denote the augmentation ideal of $\Omega(G)$, i.e., the kernel of the augmentation map. For $n\geq 1$, $\Omega(G)/I_G^n$ is a finite dimensional $\F_p$-vector space. 

\begin{definition}\label{def of Hilbert series}
For $n\geq 1$, let $c_n(G):=\op{dim}_{\F_p} \left(I_G^{n}/I_G^{n+1}\right)$, and set $c_0(G):=1$. The Hilbert series $H(G;t)$ is a formal power series defined by $H(G;t):=\sum_{n\geq 0} c_n(G) t^n$. Setting $\rho(G):=\limsup_{n\rightarrow \infty} |c_n(G)|^{\frac{1}{n}}$, we note that the radius of convergence of $H(G;t)$ is given by \[R_G:=\begin{cases}
     {\rho(G)}^{-1} & \text{ if }\rho(G)<\infty;\\
     0 & \text{ if }\rho(G)=\infty.\\
\end{cases}\]
\end{definition}
Given $x\in G$ such that $x\neq 1$, the \emph{depth} of $x$ is defined as follows
\[\omega_G(x):=\op{max}\{n\mid x-1\in I_G^n\}.\]We set $\omega_G(1):=\infty$.

\begin{definition}The \emph{Zassenhaus filtration} is defined as follows
\[G_n:=\{g\in G\mid \omega_G(g)\geq n\}.\]
\end{definition}
The sequence $G_n$ is a sequence of open normal subgroups of $G$. The quotient $G_n/G_{n+1}$ is a finite dimensional $\F_p$ vector space, set $a_n(G):=\op{dim}_{\F_p}\left(G_n/G_{n+1}\right)$.
The quantity $\rho(G)$ measures the order of exponential growth of the quotients $\left(I_G^n/I_G^{n+1}\right)$ as $n\rightarrow \infty$. These numbers are related to the growth of groups in the Zassenhaus filtration. There is an explicit relationship between the numbers $\{a_n(G)\}$ and $\{c_n(G)\}$, established by Min\'a\v c, Rogelstad and T\^ an \cite{minavc2016dimensions}.

\par Assume that $G$ is a finitely presented pro-$p$ group, and let 
\[\begin{split}& d(G):=\op{dim}_{\F_p} H^1(G, \F_p),\\ 
& r(G):=\op{dim}_{\F_p} H^2(G, \F_p).\end{split}\]
\begin{remark}\label{boring remark}
    We identify $H^1(G, \F_p)$ with $\op{Hom}\left(\frac{G}{[G,G]G^p}, \F_p\right)$, and thus,
\[d(G)=\dim_{\F_p}\left(\frac{G}{[G,G]G^p}\right).\] If $\mathcal{G}$ is a quotient of $G$, then, it is clear that $d(\mathcal{G})\leq d(G)$.
\end{remark}
Consider a minimal presentation
\begin{equation}\label{minpres}1\rightarrow R\rightarrow F\xrightarrow{\varphi} G\rightarrow 1\end{equation} of $G$. Here, $F=\langle \sigma_1, \dots, \sigma_d\rangle$ is a free pro-$p$ group generated by $d=d(G)$ elements. The subgroup $R$ of $F$ is the normal subgroup $\langle \rho_1, \dots, \rho_r\rangle^{\op{Norm}}$ generated by $r=r(G)$ elements. Given a choice of minimal presentation of $G$, we define the associated Golod-Shafarevich polynomials. Let $\Omega(F)$ be the mod-$p$ Iwasawa algebra associated to $F$, and $I_F$ be the augmentation ideal of $\Omega(F)$. By a theorem of Lazard \cite{lazard1965groupes}, the algebra $\Omega(F)$ is isomorphic to the algebra of power series $\F_p\langle \langle u_1, \dots, u_d\rangle \rangle$. Here $u_i$ is identified with the element $(\sigma_i-1)$. The augmentation ideal $I_F$ is generated by $u_1, \dots, u_d$. Let $\omega_F$ be the depth function on $F$, defined by setting 
\[\omega_F(x):=\op{max}\{n \mid x-1\in I_F^n\},\] for $x\neq 1$, and $\omega_F(1):=\infty$. The map $\varphi$ induces a surjection $\Omega(F)\rightarrow \Omega(G)$, whose kernel we denote by $J$. Identify $\Omega(G)$ with the quotient $\Omega(F)/J$ and $I_G$ with $I_F/J$. The depth filtration $\omega_G$ is related to $\omega_F$ as follows
\[\omega_G(x)=\op{max}\{\omega_F(y)\mid \varphi(y)=x\},\] cf. \cite[Lemma 3.4 and Theorem 3.5, pp. 204-205]{lazard1965groupes} for further details.

\begin{proposition}\label{prop 2.1}
The depth function $\omega_F: F\rightarrow \Z_{\geq 1}\cup \{\infty\}$ and associated filtration $\{G_n\}_n$ satisfies the following properties 
\begin{enumerate}
\item for $x\in F$, $\omega_F(x^p)\geq p\omega_F(x)$. Therefore, if $x\in G_n$, then, $x^p\in G_{np}$.
\item For $x\in G_n$ and $y\in G_m$, one has $[x,y]\in G_{n+m}$. 
\item For $x,y\in G$, we have that $\omega_F(xy)\geq \op{min}(\omega_F(x), \omega_F(y))$. 
\end{enumerate}
\end{proposition}
\begin{proof}
The above mentioned result is \cite[Proposition 1 and 2]{hajir2021cutting}, also see \cite[Section 7.4]{koch2002galois}. 
\end{proof}

\begin{definition}The Golod-Shafarevich polynomial associated with the minimal presentation \eqref{minpres} is defined as follows 
\[P(G;t):=1-dt+\sum_{i=1}^{r}  t^{\omega_F(\rho_i)}.\]
\end{definition}
Note that since the filtration is minimal, the depth $\omega_F(\rho_i)\geq 2$, cf. \cite[p.3, l. -9]{hajir2020infinite}. 
\begin{theorem}[Vinberg's criterion]\label{GS inequality}Let $G$ be a pro-$p$ group and let $H(G;t)$ be the Hilbert series associated to $G$. Choose a minimal presentation \eqref{minpres} of $G$ and let $P(G;t)$ be the associated Golod-Shafarevich polynomial. Recall that $R_G$ is the radius of convergence of $H(G;t)$, and set $R_G':=\op{min}\{1, R_G\}$. Then, for any value $t\in (0, R_G')$, the following inequality is satisfied
\[H(G;t) P(G;t)\geq 1.\]
\end{theorem}
\begin{proof}
This result is well known. We refer to \cite[Theorem 2.1]{ershov2012golod} and \cite{vinberg1965theorem} for a proof of the above statement.
\end{proof}
The next result shows that if $P(G;t)$ takes a negative value $t_0$ in the interval $(0,1)$, then, $\rho(G)>1$. Furthermore, it gives us a lower bound for $\rho(G)$, and can therefore be viewed as a refinement of the Golod-Shafarevich-Vinberg criterion (cf. \cite[Theorem 1.2]{hajir2020infinite}).
\begin{proposition}\label{lower bound on rho(G)}
Suppose that for some $t_0\in (0,1)$, the value $P(G;t_0)$ is negative. Then, the order of growth satisfies the lower bound $\rho(G)\geq \frac{1}{t_0}$. 
\end{proposition}
\begin{proof}
Let $R_G$ be the radius of convergence of the Hilbert series $H(G;t)$ and set $R_G':=\op{min}\{1, R_G\}$. By definition, the coefficients of $H(G;t)$ are all non-negative. Then, for $t\in (0,R_G')$, the series $H(G;t)$ converges absolutely to a positive value. By Vinberg's criterion, $P(G;t)H(G;t)\geq 1$ for all $t\in (0, R_G')$. Since $P(G;t_0)<0$, it follows that $H(G;t)$ does not converge absolutely at $t=t_0$. In other words, $t_0$ lies outside the domain of convergence. Therefore, $t_0\geq R_G'$. Since $t_0<1$, it follows that $R_G=R_G'<1$ and therefore, $t_0\geq R_G=\rho(G)^{-1}$. Therefore, we conclude that $\rho(G)\geq \frac{1}{t_0}$. 
\end{proof}

\begin{definition}\label{def of GS groups}
    A pro-$p$ group $G$ is said to be a \emph{Golod-Shafarevich group} if for some minimal presentation 
    \[1\rightarrow R\rightarrow F\xrightarrow{\varphi} G\rightarrow 1,\]
    there is a value $t_0\in (0,1)$ so that $P(G;t_0)<0$. Proposition \ref{lower bound on rho(G)} shows that if $G$ is a Golod-Shafarevich group with $P(G;t_0)<0$, then the order of exponential growth satisfies $\rho(G)\geq t_0^{-1}>1$. In particular $G$ is infinite since the numbers $c_n(G)$ grow at an exponential rate as $n\rightarrow \infty$.
\end{definition}

At this point, we introduce a new definition which measures the \emph{size} of a Golod-Shafarevich group. Let $G$ be a Golod-Shafarevich group and 
\[1\rightarrow R\rightarrow F\xrightarrow{\varphi} G\rightarrow 1\] be a minimal presentation of $G$. We consider quotients $G'$ of $G$ such that $d(G')=d(G)$. Suppose that $G'=G/\langle x_1, \dots, x_{m}\rangle^{\op{Norm}}$. Then, following \cite[p.4, ll. 10-17]{hajir2020infinite}, we get a new minimal presentation as follows. Lift each $x_i$ to $y_i\in F$, and let $R'=\langle \rho_1, \dots, \rho_r, y_1, \dots, y_{m}\rangle^{\op{Norm}}$ to be the normal subgroup of $F$ generated by $R$ and the elements $y_1, \dots, y_m$. Since it is assumed that $d(G')=d(G)$, this gives a minimal presentation of $G'$. In particular, note that $\omega_F(y_i)\geq 2$ for all $i=1, \dots, m$. Following \emph{loc. cit.}, we say that we have \emph{cut} the group $G$ by the elements $y_1, \dots, y_m$. Let $\varphi':F\rightarrow G'$ be the composite of $\varphi$ with the quotient map $G\rightarrow G'$. With respect to the new presentation
\[1\rightarrow R'\rightarrow F\xrightarrow{\varphi'}  G'\rightarrow 1,\] we find that for $t\in (0,1)$,  
\[P(G';t)\leq 1-d(G)t+r(G)t^2+\sum_{i=1}^m t^{\omega_F(y_i)}.\] 

\begin{definition}
    Let $G$ be a Golod-Shafarevich group, and choose a minimal presentation for $G$ such that $P(G;t)$ attains a negative value on $(0,1)$. We say that a set $\{y_1, \dots, y_m\}\subset F$ is an admissible cutting datum if $\omega_F(y_i)\geq 2$ for all $i$.
\end{definition}

Given an admissible cutting datum $\{y_1,\dots, y_m\}$, set $x_i:=\varphi(y_i)$ and \[G':=G/\langle x_1, \dots, x_m\rangle^{\op{Norm}}.\]

\begin{definition}\label{defn size}
    Let $G$ be a Golod-Shafarevich group, and choose a minimal presentation \[\mathfrak{M}: 1\rightarrow R\rightarrow F\rightarrow G\rightarrow 1\] for $G$ such that $P(G;t)$ attains a negative value on $(0,1)$. We define $m_{\mathfrak{M}}(G)$ to be the smallest value $m\in \Z_{\geq 0}$ such that there exists an admissible cutting datum $\{y_1, \dots, y_{m+1}\}$, such that $P(G';t)\geq 0$ for all $t\in (0, 1)$. If no such $m$ exists, we set $m_{\mathfrak{M}}(G):=\infty$. We set 
    \[m(G):=\op{min} \{m_{\mathfrak{M}}(G)\mid \mathfrak{M}\},\]
    where the minimum is taken over all minimal presentations $\mathfrak{M}$ of $G$. 
\end{definition}

\begin{remark}
    We note that the definition of $m(G)$ does not, to our knowledge, appear in the literature prior to this. 
\end{remark}
Thus, for any admissible cutting datum $\{y_1, \dots, y_k\}$, with $k\leq m(G)$, the associated quotient $G'$ is a Golod-Shafarevich group, and thus in particular is infinite. Like $\rho(G)$, the number $m(G)$ gives one a measure of the size of a Golod-Shafarevich group $G$. 
\section{Growth asymptotics for split prime $\Z_p$-extensions}\label{s 3}

\par Let $\dL$ be a number field which satisfies the conditions of Theorem \ref{thm 1.2}. In this section, $S_p$ (resp. $S_\infty$) denote the primes of $\dL$ that lie above $p$ (resp. $\infty$). Set $S$ to denote the union. Let $L/\dL$ be a Galois number field extension and let $S(L)$ (resp. $S_p(L)$) be the set of places of $L$ which lie above $S$ (resp. $S_p$). We shall denote by $L_S$ the maximal pro-$p$ algebraic extension of $L$ in which the primes $w\notin S(L)$ are unramified. We denote by $\op{G}_S(L)$ the Galois group $\op{Gal}(L_S/L)$. Set $H_S(L)$ to be the maximal abelian unramified extension of $L$ in which the places $v\in S(L)$ are split. The $S$ class group of $L$ is the Galois group $\op{Cl}_S(L):=\op{Gal}(H_S(L)/L)$, and is identified with a quotient of the class group $\op{Cl}(L)$. For $v\in S_p$, let $(e_v(L/\dL), f_v(L/\dL), g_v(L/\dL))$ denote the ramification invariants of $v$ in $L$. In other words, $v$ splits into $g_v(L/\dL)$ primes in $\cO_L$, each with ramification index $e_v(L/\dL)$. The quantity $f_v(L/\dL):=[\cO_{L}/w: \cO_K/v]$ is independent of the choice of prime $w|v$ of $L$. We note that $e_v(L/\dL) f_v(L/\dL) g_v(L/\dL)=[L:\dL]$. 
\begin{theorem}\label{thm 2.2}
Let $L/\dL$ be a finite Galois extension. With respect to the above notation, the following relations hold
\[\begin{split}&\op{dim}_{\F_p}H^1(\op{G}_S(L), \F_p)=\sum_{i=1}^g g_{\p_i}(L/\dL)+\frac{[L:\Q]}{2}+\op{dim} \left(\op{Cl}_S(L)\otimes \F_p\right),\\
& \op{dim}_{\F_p}H^2(\op{G}_S(L), \F_p)=\sum_{i=1}^g g_{\p_i}(L/\dL)-1+\op{dim} \left(\op{Cl}_S(L)\otimes \F_p\right).
\end{split}\]
\end{theorem}
\begin{proof}
The above is a special case of \cite[Theorem 10.7.3]{neukirch2013cohomology}. 
\end{proof}
We order the set of primes above $p$ such that $\p_i$ is infinitely decomposed $\dL_\infty$ for $i\in [1, T_1]$ and infinitely ramified for $i\in [T_1+1, T_2]$. The primes $\p_i$ for $i\in [T_2+1, g]$ are unramified and finitely decomposed in $\dL_\infty$. We shall set $g_i(n)$, $e_i(n)$ and $f_i(n)$ to denote $g_{\p_i}(\dL_n/\dL)$, $e_{\p_i}(\dL_n/\dL)$ and $f_{\p_i}(\dL_n/\dL)$ respectively. We find that

\begin{equation}
\label{ramification invariants}\begin{split}& g_{1}(n)=p^{n} , e_{1}(n)=f_{1}(n)=1,\\
& g_{i}(n)\leq p^{n} \text{ for }i\in [2, T_1],\\
& g_{i}(n)=O(1)\text{ for }i>T_1.
\end{split}\end{equation}
We set 
\[d(n):=\op{dim}_{\F_p}H^1(\op{G}_S(\dL_n), \F_p)\text{ and }r(n):=\op{dim}_{\F_p}H^2(\op{G}_S(\dL_n), \F_p).\]

Before stating the next result, we clarify some standard conventions with regard to our notation. Let $f, g, h: \Z_{\geq 0}\rightarrow \mathbb{R}_{\geq 0}$ be non-negative functions. We write $f(n)=g(n)+O(h(n))$ to mean that there is a positive constant $C>0$, independent of $n$, such that for all $n\in \mathbb{N}$, $|f(n)-g(n)|\leq C h(n)$. We write $f(n)\leq g(n) +O(h(n))$ (resp. $f(n)\geq g(n) +O(h(n))$) to mean that $f(n)\leq g(n)+C h(n)$ (resp. $f(n)\geq g(n)-C h(n)$) for some absolute constant $C>0$.

\begin{corollary}\label{corollary bounds on d(n) and r(n)}
    The following relations hold
    \[\begin{split} & d(n)\geq  p^n\left(1+\frac{[\dL:\Q]}{2}\right)\\ 
 & d(n)\leq p^n\left(T_1+\frac{[\dL:\Q]}{2}+\mu(\dL_\infty/\dL)\right)+o(p^n),\\ 
 &r(n)=  O(p^n).
\end{split}\]
\end{corollary}

\begin{proof}
   It follows from Iwasawa's theorem that $\op{dim}_{\F_p}\left(\op{Cl}(\dL_n)\otimes \F_p\right)\leq \mu(\dL_\infty/\dL)p^n +o(p^n)$. The result follows therefore as an immediate consequence of Theorem \ref{thm 2.2} and \eqref{ramification invariants}.
\end{proof}

\par For $n\geq 0$ and $v\in S_p(\dL_n)$, set $n_v:=\op{dim}_{\F_p}\left(\dL_{n,v}^\times\otimes \F_p\right)$. We choose an embedding of \[\iota:\widebar{\dL}_n\hookrightarrow \widebar{\dL}_{n, v},\] or equivalently, a prime $v'$ of $\widebar{\dL}_n$ that lies above $v$. We shall refer to the decomposition group of $v'|v$ as the decomposition group at $v$. The inclusion $\iota$ prescribes an inlusion of $\op{Gal}(\bar{\dL}_{n, v}/\dL_{n, v})$ into $\op{G}_{\dL_n}$, and the image of this embedding is identified with the decomposition group at $v$. Note that the pro-$p$ completion of the decomposition group at $v$ is generated by $n_v$ elements. For ease of notation, we set $d_i:=[\dL_{\p_i}:\Q_p]$. We note that since $\mu_p$ is contained in $\dL$,
\begin{equation}\label{n_v def equation}\begin{split}n_v=\op{dim}_{\F_p}\left(\dL_{n,v}^\times\otimes \F_p\right)=& [\dL_{n,v}:\Q_p]+2,\\=& \left(e_{i}(n)f_{i}(n)[\dL_{\p_i}:\Q_p]+2\right)\text{ if }v|\p_i,\\
& \begin{cases}= (d_1+2)\text{ if }v|\p_1, \\
\leq (p^{n}d_i+2)\text{ if }v|\p_i\text{ for }i\geq 2.\end{cases} \\
\end{split}\end{equation} 
Let $\widetilde{\cL}_n$ be the maximal pro-$p$ extension of $\dL_n$ unramified at all primes $v\notin S(\dL_n)$. Denote by $\cL_n$ the maximal pro-$p$ extension of $\dL_n$ unramified at all primes $v\notin S(\dL_n)$ and such that for all primes $v\in S_p(\dL_n)$, the decomposition group is abelian. For $k\geq 0$, let $\cF^{[k]}(\dL_n)$ be the maximal subfield of $\cL_n$ in which all the elements in the decomposition groups at primes $v\in S_p(\dL_n)$ have order dividing $p^k$. It is thus understood that $\cF^{[0]}(\dL_n)$ is the maximal unramified pro-$p$ extension of $\dL_n$ in which all primes $v\in S(\dL_n)$ are completely split. We have the following containments 
\[\dL_n\subseteq \cF^{[0]}(\dL_n)\subseteq \cF^{[1]}(\dL_n)\subseteq\dots \subseteq \cF^{[k]}(\dL_n)\subseteq\cF^{[k+1]}(\dL_n)\subseteq\dots \subseteq \cL_n\subseteq \widetilde{\cL}_n.\]
Set $\widetilde{\cG}_n$, $\cG_n$ and $\cG_n^{[k]}=\cG^{[k]}(\dL_n)$ denote the Galois groups $\op{Gal}(\widetilde{\cL}_n/\dL_n)$, $\op{Gal}({\cL}_n/\dL_n)$ and $\op{Gal}(\cF^{[k]}(\dL_n)/\dL_n)$ respectively. We begin with a minimal filtration 
\begin{equation}\label{filtration 1}1\rightarrow R\rightarrow F\xrightarrow{\varphi} \widetilde{\cG}_n\rightarrow 1\end{equation}of $\widetilde{\cG}_n$. The group $\cG_n$ is obtained from $\widetilde{\cG}_n$ upon cutting by the commutators of the generators of all decomposition groups for primes $v\in S_p(\dL_n)$. In greater detail, for each prime $v\in S_p(\dL_n)$, let $z_1^{(v)}, \dots, z_{n_v}^{(v)}$ be a set of elements in $F$ mapping to a set of generators of the decomposition group at $v$. For $1\leq i<j\leq n_v$, we let $z_{i,j}^{(v)}\in F$ denote the commutator $[z_i^{(v)}, z_j^{(v)}]$. It follows from Proposition \ref{prop 2.1} that \begin{equation}\label{omega eqn}\omega(z_{i,j}^{(v)})\geq \omega(z_i^{(v)})+\omega(z_j^{(v)})\geq 2.\end{equation} We represent $\cG_n$ as the quotient  \[\cG_n=\frac{\widetilde{\cG}_n}{\langle \varphi\left(z_{i,j}^{(v)}\right)\mid 1\leq i<j \leq n_v, v\in S_p(\dL_n)\rangle^{\op{Norm}}},\] and we obtain a new presentation
\begin{equation}\label{filtration 2}1\rightarrow R'\rightarrow F\rightarrow \cG_n\rightarrow 1,\end{equation}where, 
\[R'=R\langle z_{i,j}^{(v)}\mid 1\leq i<j \leq n_v, v\in S_p(\dL_n)\rangle^{\op{Norm}}.\]
Going modulo the $p^k$-th powers of all generators of decomposition groups, we obtain a presentation 
\begin{equation}\label{filtration 3}1\rightarrow R''\rightarrow F\rightarrow \cG_n^{[k]}\rightarrow 1,\end{equation}
where, 
\[R''=R'\langle \left(z_{i}^{(v)}\right)^{p^k}\mid 1\leq i\leq n_v, v\in S_p(\dL_n)\rangle^{\op{Norm}}.\]
Set $\widetilde{\cP}_n(t):=\cP(\widetilde{\cG}_n,t)$, $\cP_n(t):=\cP(\cG_n,t)$, and $\cP_n^{[k]}(t):=\cP(\cG_n^{[k]},t)$ to be the Golod-Shafarevich polynomials associated with the minimal filtrations \eqref{filtration 1}, \eqref{filtration 2} and \eqref{filtration 3} respectively. It follows from Remark \ref{boring remark} that 
\begin{equation}\label{some boring inequalities}d(n)=d(\widetilde{\cG}_n)\geq d(\cG_n)\geq d(\cG_n^{[k]})\end{equation} for all $k\geq 0$. Since we successively cut by elements with depth $\geq 2$, we have that 
\[d(n)=d(\widetilde{\cG}_n)=d(\cG_n)=d(\cG_n^{[k]}),\]
(cf. \cite[Proposition 5]{hajir2021cutting}).

\begin{proposition}\label{prop 3.4}
    For $t\in (0,1)$,
    $\cP_n(t)\leq 1-D_n t+R_n t^2$, where $D_n, R_n$ are constants satisfying
    \[\begin{split}D_n\geq  & p^n\left(1+\frac{[\dL:\Q]}{2}\right),\\
    D_n\leq  & p^n\left(T_1+\frac{[\dL:\Q]}{2}+\mu(\dL_\infty/\dL)\right)+o(p^n),\\
    R_n= & p^{2n}\left(\frac{\sum_{i=2}^gd_i^2}{2}\right)+O(p^n).
    \end{split}\]
\end{proposition}
\begin{proof}
    The group $\cG_n$ is obtained from $\widetilde{\cG}_n$ by quotienting by the commutators of each of the decomposition groups at primes $v\in S_p(\dL_n)$. At each prime $v$, we cut by $\binom{n_v}{2}$ elements $\{z_{i,j}^{(v)}\mid 1\leq i<j\leq n_v , v\in S_p(\dL_n)\}$, where we recall that an upper bound for $n_v$ is given by \eqref{n_v def equation}. Furthermore, it follows from \eqref{omega eqn} that $\omega(z_{i,j}^{(v)})\geq 2$. Therefore, we find that for $t\in (0,1)$, 
\[\begin{split}
\cP_n(t) &  \leq \widetilde{\cP}_n(t)+\left(\sum_{v\in S_p(\dL_n)} \binom{n_v}{2}\right) t^2 \\
&\leq \widetilde{\cP}_n(t)+ \left( p^n \binom{d_1+2}{2} +\sum_{i=2}^g \binom{p^n d_i+2}{2}\right) t^2\\
& \leq 1-d(n)t +\left( r(n)+ p^n \binom{d_1+2}{2} +\sum_{i=2}^g \binom{p^n d_i+2}{2}\right) t^2 \\
&\leq 1-D_nt+R_nt^2, \\
\end{split}\]
where 
\[\begin{split} D_n &:=d(n),\\
 R_n &:=r(n)+ p^n \binom{d_1+2}{2} +\sum_{i=2}^g \binom{p^n d_i+2}{2}\\
 & = p^{2n}\left(\frac{\sum_{i=2}^gd_i^2}{2}\right)+O(p^n).\\
\end{split}
\]
Here, we have invoked the bounds for $d(n)$ and $r(n)$ from Corollary \ref{corollary bounds on d(n) and r(n)}.
\end{proof}

\begin{proof}[Proof of Theorem \ref{thm 1.2}]
\par It follows from Proposition \ref{prop 3.4} that for $t\in (0,1)$,
\[\cP_n(t)\leq 1-D_n t+R_n t^2,\]
where 
\begin{equation}\label{bounds for D_n and R_n}\begin{split}D_n\geq & p^n\left(1+\frac{[\dL:\Q]}{2}\right)\\
    D_n \leq  & p^n\left(T_1+\frac{[\dL:\Q]}{2}+\mu(\dL_\infty/\dL)\right)+o(p^n),\\
    R_n= & p^{2n}\left(\frac{\sum_{i=2}^gd_i^2}{2}\right)+O(p^n).
    \end{split}\end{equation}
 For $k\geq 1$, the group $\cG_n^{[k]}$ is obtained from $\cG_n$ by quotienting by the $p^k$-th powers of the generators of the decomposition groups at the primes $v\in S_p(\dL_n)$. The total number of elements we quotient by is 
\begin{equation}\label{R_n'}\begin{split} R_n' & :=\sum_{v\in S_p(\dL_n)} n_v,\\
& \leq p^n(d_1+2)+\sum_{i=2}^g (p^nd_i+2),\\
&= p^n(\sum_{i=1}^g d_i +2)+O(1).
\end{split}\end{equation}
Let $y_1, \dots, y_{R_n'}$ be the elements that generate the decomposition groups at the primes $v\in S_p(\dL_n)$. By Proposition \ref{prop 2.1}, we have that 
\[\omega_F(y_i^{p^k})= p^k\omega_F(y_i)\geq p^k.\]
Therefore, for $t\in (0,1)$, we find that
\[\cP_n^{[k]}(t)\leq \cP_n(t)+R_n't^{p^k}\leq \cP_n(t)+R_n't^{p}.\]
\par Consider the polynomial $Q(t):=1-D_nt +R_nt^2+R_n't^{p}$. Setting $t_n:=\frac{D_n}{2R_n}$, note that the quadratic part $1-D_nt +R_nt^2$ of $Q(t)$ attains it minimum value at $t_n$. We find that 
\[Q(t_n)=1-\frac{D_n^2}{4R_n}+\frac{R_n' D_n^p}{2^p R_n^p} .\] 
We require that $Q(t_n)<0$, i.e., 
\[D_n^2>4R_n+\frac{R_n'D_n^p}{2^{p-2}R_n^{p-1}},\] i.e., 
\begin{equation}\label{inequality main R_n D_n}
2^{p-2} R_n^{p-1} D_n^2>2^p R_n^p+R_n' D_n^p.\end{equation}
It follows from the bounds \eqref{bounds for D_n and R_n} that as $n\rightarrow \infty$,
\begin{equation}\label{A and B equations}
    \begin{split}
        & D_n^2\geq   Ap^{2n};\\
        & R_n=Bp^{2n}+O(p^n);
    \end{split}
\end{equation}
where \[\begin{split} & A=\left(1+\frac{[\dL:\Q]}{2}\right)^2 \\
& B=\left(\frac{\sum_{i=2}^gd_i^2}{2}\right).\end{split}\]
Note that by \eqref{R_n'}, $R_n'=O(p^n)$. We estimate both sides of the inequality \eqref{inequality main R_n D_n}. Therefore, we obtain the following asymptotic estimates
\begin{equation}\label{random equation}
    \begin{split}
        & 2^{p-2} R_n^{p-1} D_n^2\geq 
 2^{p-2}B^{p-1}A p^{2pn}+O(p^{(2p-1)n});\\
 & R_n'D_n^p=O(p^{(p+1)n});\\
 & 2^p R_n^p=2^p B^pp^{2np}+O(p^{(2p-1)n}).
    \end{split}
\end{equation}
Since $2p-1\geq p+1$, it follows that $R_n' D_n^p=O(p^{(2p-1)n})$.
By condition \eqref{p4 of thm 1.2} of Theorem \ref{thm 1.2},
\[\frac{2^{p-2}B^{p-1}A}{2^pB^p}=\frac{A}{4 B}=\frac{\left(2+[\dL:\Q]\right)^2}{8\left(\sum_{i=2}^g d_i^2\right)}>1.\]
Therefore, for large enough values of $n$, we have that $Q(t_n)<0$. By the estimates in the statement of Proposition \ref{prop 3.4},
\begin{equation}\label{R_n and D_n}\begin{split}D_n\leq  & p^n\left(T_1+\frac{[\dL:\Q]}{2}+\mu(\dL_\infty/\dL)\right)+o(p^n)\\
    R_n= & p^{2n}\left(\frac{\sum_{i=2}^gd_i^2}{2}\right)+O(p^n).
    \end{split}\end{equation} Therefore, in particular, we find that $t_n\in (0,1)$ for all large enough values of $n$. Since $\cP_n^{[k]}(t)\leq Q(t)$ for all $t\in (0,1)$, it follows that $\cP_n^{[k]}(t_n)<0$ for all large enough values of $n$, and all values of $k\geq 1$. 

With respect to notation above, 
\[t_n^{-1}=\frac{2R_n}{D_n}.\]

Let $C>0$ be a constant for which
\[C<\frac{2\left(\sum_{i=2}^g d_i^2\right)}{\left(2 T_1+[\dL:\Q]+2\mu(\dL_\infty/\dL)\right)}.\] Then, we find that for all large enough values of $n$, we find that $t_n^{-1}>Cp^n$. Since for large enough values of $n$, $\cP_n^{[k]}(t_n)<0$ and $t_n\in (0,1)$, it follows from Proposition \ref{lower bound on rho(G)} that \[\rho^{[k]}(\dL_n)\geq t_n^{-1},\] and this proves the result.
\end{proof}

\begin{proof}[Proof of Corollary \ref{main cor}]
   We show that the conditions of Theorem \ref{thm 1.2} are satisfied.  
   \begin{enumerate}
       \item It is assumed that $g=\ell-1$, in particular, $g>1$.
       \item Clearly, $\dL:=\Q(\mu_{p\ell})$ contains $\mu_p$. 
       \item Note that $d_i=(p-1)$ for all $i$. Since $\ell\geq 5$, we find that
       \[[\dL:\Q]=(\ell-1)(p-1)\geq 2(d_1+1)=2p.\]
       \item We find that $([L:\Q]+2)^2=((\ell-1)(p-1)+2)^2>(\ell-1)^2(p-1)^2$ and that $8(\sum_{i=2}^g d_i^2)=8(\ell-2)(p-1)^2$. Clearly, $\ell\geq 11$ implies that $(\ell-1)^2>8(\ell-2)$, and therefore, 
       \[([L:\Q]+2)^2>8(\sum_{i=2}^g d_i^2).\]
   \end{enumerate}
\end{proof}

Next, we set $m^{[k]}(\dL_n):=m\left(\cG_n^{[k]}\right)$, where we recall that $\cG_n^{[k]}:=\cG^{[k]}(\dL_n)$. The following result gives an asymptotic lower bound for the numbers $m^{[k]}(\dL_n)$, as $n\rightarrow \infty$.

\begin{theorem}\label{thm 3.6}
Let $p$ be a prime number and $\dL$ be a number field for which the conditions of Theorem \ref{thm 1.2} are satisfied.
Then, there exists a constant $c>0$ (independent of $n$ and $k$) and $n_0\in \Z_{\geq 0}$, such that for all $n\geq n_0$ and $k\geq 1$, we have that
\[m^{[k]}(\dL_n)\geq cp^n.\]
\end{theorem}
\begin{proof}
    We set $m_n:=m\left(\cG_n^{[k]}\right)$, and assume without loss of generality that $m_n<\infty$. We obtain a lower bound on $m_n$. Choose a minimal presentation 
    \[1\rightarrow R\rightarrow F \xrightarrow{\varphi} \cG_n^{[k]}\rightarrow 1\] such that there exists $t_0\in (0,1)$ for which $\cP_n^{[k]}(t):=P(\cG_n^{[k]}, t_0)<0$ for the associated Golod-Shafarevich polynomial. Let $\{y_1, \dots, y_{m+1}\}\subset F$ be an admissible cutting datum. Note that by definition, $\omega_F(y_i)\geq 2$ for all $i$. Setting $x_i:=\varphi(y_i)$, denote by $G'$ the quotient $\cG_n^{[k]}/\langle x_1, \dots, x_{m+1}\rangle^{\op{Norm}}$. By definition, the admissible cutting datum $\{y_1, \dots, y_{m+1}\}$ can be chosen so that $P(G',t)\geq 0$ for all $t\in (0,1)$. From the proof of Theorem \ref{thm 1.2}, we have that $\cP_n^{[k]}(t)\leq Q(t)$, where $Q(t):=1-D_n t+R_n t^2+R_n' t^{p^k}$. For $t\in (0,1)$, we find that 
    \[\begin{split}P(G',t) & \leq \cP_n^{[k]}(t)+(m+1)t^2 \\
    & \leq Q(t)+ (m+1)t^2 \\
    & = 1-D_n t+(R_n+m+1) t^2+R_n' t^{p^k}\\ 
    & \leq  1-D_n t+(R_n+m+1) t^2+R_n' t^{p}.\end{split}\]
    In what follows, we set $a_n:=\frac{ D_n}{2(R_n+m+1)}$. From the estimates \eqref{R_n and D_n}, it follows that $a_n\in (0,1)$ for large enough values of $n$. Therefore, assume that $n$ is large enough so that $a_n\in (0,1)$. Then, $P(G',a_n)\geq 0$, i.e.,
    \[1-D_n a_n+(R_n+m+1) a_n^2+R_n' a_n^p\geq 0.\]
    We find that 
    \[\begin{split} & 1-D_n a_n+(R_n+m+1) a_n^2+R_n' a_n^p \\
    = & \frac{2^p(R_n+m+1)^p-2^{p-2}(R_n+m+1)^{p-1}D_n^2+R_n'D_n^p}{2^p(R_n+m+1)^p}
    \end{split}\]
    This implies that 
    \[\begin{split}(R_n+m+1) &\geq \frac{D_n^2}{4}-\frac{R_n' D_n^p}{2^p(R_n+m+1)^{p-1}};\\
    & \geq \frac{D_n^2}{4}-\frac{R_n' D_n^p}{2^pR_n^{p-1}},
    \end{split}\]
    and therefore, we get the inequality
    \begin{equation}\label{e 3} m\geq \frac{D_n^2}{4}-R_n-\frac{R_n' D_n^p}{2^pR_n^{p-1}}-1.\end{equation}
    By \eqref{A and B equations},  
   \begin{equation}\label{e 1}\begin{split}
        & D_n^2\geq  Ap^{2n};\\
        & R_n=Bp^{2n}+O(p^n);
    \end{split}\end{equation}
where\[\begin{split} & A=\left(1+\frac{[\dL:\Q]}{2}\right)^2 \\
& B=\left(\frac{\sum_{i=2}^gd_i^2}{2}\right).\end{split}\]
Furthermore, by \eqref{random equation}, \begin{equation}\label{e 2}\frac{R_n' D_n^p}{2^pR_n^{p-1}}=O(p^{(3-p)n})=O(1).\end{equation}
Therefore, combining \eqref{e 3}, \eqref{e 1} and \eqref{e 2}, we have that 
\[m\geq (A/4-B) p^{2n}-Cp^n,\] for some large enough constant $C>0$. Recall that $m_n$ is the minimum value of $m$ over all possible minimal presentations. We choose a minimal presentation such that $m_n=m$. It follows from the condition \eqref{p4 of thm 1.2} of Theorem \ref{thm 1.2} that $A>4B$. Setting $c:=A/4-B$, the result follows.
\end{proof}

\bibliographystyle{alpha}
\bibliography{references}
\end{document}